\documentclass[11pt,a4paper]{article}
\usepackage{epsfig}
\usepackage{caption}
\usepackage{subcaption}
\usepackage{epic}
\usepackage{eepic}
\usepackage{bm}
\usepackage{amssymb}
\usepackage{amsmath}
\usepackage{amsthm}
\usepackage{caption}
\usepackage{subcaption}
\usepackage{colortbl}
\usepackage{color}
\usepackage{hyperref}
\usepackage{float}

\newtheorem{theorem}{Theorem}[section]

\newtheorem{lemma}[theorem]{Lemma}
\newtheorem{corollary}[theorem]{Corollary}

\newcommand{\req}[1]{(\ref{#1})}

\newcommand{\alg}[1]{{\mathbf #1}}

\newcommand{\nmbr}[1]{\mathbb{#1}}
\newcommand{\bc}[1]{{\mathcal #1}}
\newcommand{\dual}[2]{\left\langle #1,#2 \right\rangle}

\title{Numerical approximation of the spectrum of self-adjoint continuously invertible operators}

\author{
	Tom{\'a}{\v s} Gergelits\thanks{Faculty of Mathematics and Physics, Charles University, Prague,
                   Czech Republic. Email: gergelits@karlin.mff.cuni.cz. \textcolor{black}{This author was supported
                   by Lawrence Livermore National Security, LLC Subcontract Award B639388 under Prime
                   Contract No. DE-AC52-07NA27344.}
}
	\and
	Bj{\o}rn Fredrik Nielsen\thanks{Faculty of Science and Technology, Norwegian University of Life Sciences,
                   P.O. Box 5003, NO-1432 {\AA}s, Norway. Email: bjorn.f.nielsen@nmbu.no. Nielsen's work
                   was supported by The Research Council of Norway, project number 239070.}
	\and
	Zden{\v e}k Strako{\v s}\thanks{Faculty of Mathematics and Physics, Charles University, Prague, Czech Republic.
                   Email: strakos@karlin.mff.cuni.cz. Supported by the Grant Agency of the Czech Republic
                   under the contract No. 17-04150J.}
}

\begin{document}
\maketitle

\begin{abstract}
This paper deals with the generalized spectrum of continuously invertible linear operators defined on infinite dimensional Hilbert spaces.
More precisely, we consider two bounded, coercive, and  self-adjoint operators
$\bc{A, B}: V\mapsto V^{\#}$, where $V^{\#}$ denotes the dual of $V$,
and investigate the conditions under which the whole spectrum of $\bc{B}^{-1}\bc{A}:V\mapsto V$
can be approximated to an arbitrary accuracy by the eigenvalues
of the finite dimensional discretization $\bc{B}_n^{-1}\bc{A}_n$. Since $\bc{B}^{-1}\bc{A}$ is continuously invertible,
such an investigation cannot use the concept of uniform (normwise) convergence, and it relies instead
on the pointwise (strong) convergence of $\bc{B}_n^{-1}\bc{A}_n$ to $\bc{B}^{-1}\bc{A}$.

The paper is motivated by operator preconditioning which is employed in the numerical solution of boundary value problems. In this context,
$\bc{A}, \bc{B}: H_0^1(\Omega) \mapsto H^{-1}(\Omega)$  are the standard integral/functional representations
of the differential operators $ -\nabla \cdot (k(x)\nabla u)$ and $-\nabla \cdot (g(x)\nabla u)$, respectively, and $k(x)$ and
$g(x)$ are scalar coefficient functions. The investigated question differs
from the eigenvalue problem studied in the numerical PDE literature which is based on the approximation
of the eigenvalues within the framework of compact operators.

This work follows the path started by the two recent papers published in
[SIAM J. Numer. Anal., 57 (2019), pp.~1369-1394 and 58 (2020), pp.~2193-2211] and addresses one of the open questions
formulated at the end of the second paper.

\end{abstract}


\section{Introduction.}
\label{introduction}

Extending the path of research started in~\cite{Nie09,Ger19,Ger20}, this paper will
consider the differential operators  $-\nabla \cdot (k(x)\nabla u)$ and $-\nabla \cdot (g(x)\nabla u)$
on the open and bounded Lipschitz domain $\Omega\subset \nmbr{R}^2$, where the scalar functions
$g(x)$ and $k(x)$ are uniformly positive and continuous throughout the closure $\overline{\Omega}$.
The associated operator representations
$\bc{A}, \, \bc{B}: H_0^1(\Omega) \mapsto H^{-1}(\Omega)$, are given by
\begin{align}\label{eq:A}
& \langle \mathcal{A} u, v \rangle = \int_{\Omega} k(x) \nabla u  \cdot  \nabla v,  \quad u,  v \in H_0^1(\Omega), \\
\label{eq:B}
& \langle \mathcal{B} u, v \rangle = \int_{\Omega} g(x) \nabla u  \cdot  \nabla v, \quad u,  v \in H_0^1(\Omega).
\end{align}
In the first part of this paper we characterize the spectrum of the preconditioned operator
\begin{equation}\label{eq:operator}
\bc{B}^{-1}\bc{A}: H_0^1(\Omega)  \to H_0^1(\Omega),
\end{equation}
defined as the complement of the resolvent set, i.e.,
\begin{equation}\label{eq:spectrum}
\mathrm{sp}(\mathcal{B}^{-1} \mathcal{A})
:= \left\{ \lambda \in \mathbb{C}; \,  \lambda \mathcal{I} - \mathcal{B}^{-1} \mathcal{A}
\mbox{ does not have a bounded inverse} \right\}.
\end{equation}
More specifically, we prove that
$$
 \mathrm{sp}(\mathcal{L}^{-1} \mathcal{A}) =
 \left[\inf_{x\in\overline{\Omega}}\
 \frac{k(x)}{g(x)},\,\sup_{x\in\overline{\Omega}}\ \frac{k(x)}{g(x)}\right].
$$

Consider a sequence of finite dimensional subspaces of $H_0^1(\Omega)$ defined via
the nodal polynomial basis functions\footnote{\textcolor{black}{As in~\cite{Ger19}, we consider
conforming FE methods using Lagrange elements.}} $\phi_1,\ldots,\phi_n$ with the local supports
\begin{equation}\label{eq:support}
\bc{T}_j = \mbox{supp}(\phi_j),\quad j=1,\ldots,n.
\end{equation}
The standard Galerkin finite element discretization of the operators $\bc{A}$ and $\bc{B}$
gives the matrix representations of the discretised operators $\mathcal{A}_n$ and $\mathcal{B}_n$
in terms of the basis $\phi_1,\ldots,\phi_n$,
\begin{align}\label{eq:dis:A}
	[\alg{A}_n]_{ij} &= \int_{\Omega} k(x) \nabla \phi_j\cdot \nabla\phi_i, \, \quad i,j=1,\ldots,n.\\
	[\alg{B}_n]_{ij} &= \int_{\Omega} g(x) \nabla \phi_j\cdot \nabla \phi_i,\, \quad i,j=1,\ldots,n.
\label{eq:dis:B}
\end{align}
Part one of this paper also contains an investigation of the approximation of the whole spectrum
of $\bc{B}^{-1}\bc{A}$ by the eigenvalues of the preconditioned matrices
$\alg{B}_n^{-1} \alg{A}_n$ as  $n \rightarrow \infty$.

In the second part, i.e., in section \ref{Abstract Galerkin discretization}, we generalize the results obtained for \eqref{eq:A} and \eqref{eq:B}.
More precisely, the spectrum issue is explored in terms of an abstract setting
where  $\bc{A}, \bc{B}: V\mapsto V^{\#}$ only are assumed
to be bounded, coercive, and  self-adjoint\footnote{Self-adjoint in the sense that
\textcolor{black}{$\langle \bc{A} u, v \rangle = \langle \bc{A} v, u \rangle,
\langle \bc{B} u, v \rangle = \langle \bc{B} v, u \rangle$ for all $u, v \in V$, where
$\langle\cdot,\cdot\rangle: V^\#\times V \mapsto \nmbr{R}$ is the duality pairing.
Equivalently,}
$\tau \bc{A}$, $\tau \bc{B}: V \rightarrow V$ are self-adjoint, where $\tau$ is the Riesz map.}
\textcolor{black}{linear} operators. Here,  $V^\#$
denotes the dual of $V$ consisting of all linear bounded functionals from the infinite dimensional
Hilbert space $V$ to $\nmbr{R}$. We present a condition under which the whole spectrum
of $\bc{B}^{-1}\bc{A}: V\mapsto V$, defined as the complement of the resolvent set, is approximated as
$n \rightarrow \infty$ to an arbitrary accuracy by the eigenvalues of the finite dimensional
discretizations $\bc{B}_n^{-1}\bc{A}_n$. More precisely, we will concentrate on Galerkin
discretizations using a sequence $\{V_n\}$ of subspaces
$V_n  \subset V$ satisfying the  approximation property
\begin{equation}\label{eq:approximation_property_of subspaces}
\lim_{n \rightarrow \infty} \inf_{v \in V_n} \| w - v \| = 0   \quad \mbox{for all} \;
 w \in V \,.
\end{equation}

Since $\bc{B}^{-1}\bc{A}$ is continuously invertible and $V$ is of infinite dimension, such an investigation
cannot be based on the uniform (normwise) convergence, and it relies instead upon the pointwise
(strong) convergence of $\bc{B}_n^{-1}\bc{A}_n$ to $\bc{B}^{-1}\bc{A}$.

\section{Preconditioning by Laplacian ($g(x) = 1$).}
\label{Preconditioning by Laplacian}

Considering the case $g(x) = 1$,  i.e., the preconditioner equals the operator representation
of the Laplacian $\bc{B} = \bc{L}$, the paper~\cite{Ger20} determines the spectrum of
the preconditioned operator $\bc{L}^{-1}\bc{A}$ in the following way
(for brevity we use a bit stronger assumptions than in~\cite{Ger20}  and  consider a uniformly positive
scalar coefficient function $k(x)$):

\begin{theorem}[cf.~\cite{Ger20}, Theorem~1.1]
\label{th:Laplacian_preconditioning_operator_spectrum}
Consider an open and bounded Lipschitz domain $\Omega\subset \nmbr{R}^2$.
Assume that the scalar function $k(x)$ is uniformly positive and continuous throughout
the closure $\overline{\Omega}$. Then the spectrum of the operator $\bc{L}^{-1}\bc{A}$
equals the interval
\begin{equation}\label{eq:result1-L}
 \mathrm{sp}(\mathcal{L}^{-1} \mathcal{A}) =
 \left[\inf_{x\in\overline{\Omega}}\ k(x),\,\sup_{x\in\overline{\Omega}}\ k(x)\right].
\end{equation}
\end{theorem}

\noindent
In other words, for a uniformly positive continuous function $k(x)$, the spectrum
of $\bc{L}^{-1}\bc{A}$ equals the range $k(\overline{\Omega})$.

Eigenvalues of the discretized operator $\mathcal{L}_n^{-1} \mathcal{A}_n$, that is represented
by the matrix $\alg{L}_n^{-1}\alg{A}_n$, can be approximated using the following
theorem from~\cite{Ger19}. (Here we again use assumptions conforming
to the setting in the current paper.)

\begin{theorem}[cf.~\cite{Ger19}, Theorem~3.1]
\label{th:Laplacian_preconditioning_operator_matrix}
Let $0< \lambda_1\leq\lambda_2\leq\ldots\leq\lambda_n$ be the eigenvalues
of $\alg{L}_n^{-1}\alg{A}_n$, where $\alg{A}_n$ and $\alg{L}_n$ are defined in
\eqref{eq:dis:A} and \eqref{eq:dis:B}\footnote{With $g(x) = 1$.}, respectively.
Let $k(x)$ be uniformly positive and continuous throughout the closure $\overline{\Omega}$.
Then there exist a (possibly non-unique) permutation $\pi$ such that the eigenvalues
$\lambda_{\pi(j)}$ of the matrix $\alg{L}_n^{-1}\alg{A}_n$ satisfy
\begin{equation}\label{eq:result2-L}
\lambda_{\pi(j)} \in
\left[\inf_{ x\in\bc{T}_j} k(x),\sup_{ x\in\bc{T}_j} k(x)\right],\quad j = 1,\ldots,n,
\end{equation}
where $\bc{T}_j$ is defined in~(\ref{eq:support}).
\end{theorem}

\noindent
Consequently, there is a one-to-one mapping (possibly non-unique) between the eigenvalues
of $\alg{L}_n^{-1}\alg{A}_n$ and the ranges of $k(x)$ over the supports of the individual
basis functions. With an appropriate grid refinement of the discretization, the size of the intervals
containing the individual eigenvalues of $\alg{L}_n^{-1}\alg{A}_n$ converge
linearly to zero.\footnote{\textcolor{black}{An interesting application inspired by this result
that uses a different approach is presented in~\cite{LPZ20}}}

The paper~\cite{Ger20} formulates a dual open question about the {\em distribution
of the eigenvalues of the discretized operators within the interval~(\ref{eq:result1-L}).}
Assuming in addition that $k(x) \in \bc{C}^2(\overline{\Omega})$, we will now show that
theorems~\ref{th:Laplacian_preconditioning_operator_spectrum}
and~\ref{th:Laplacian_preconditioning_operator_matrix} yield that
{\em any point in the spectrum} of the infinite dimensional operator $\bc{L}^{-1}\bc{A}$
is approximated, as $n \rightarrow \infty$, to an arbitrary accuracy
by the eigenvalues~(\ref{eq:result2-L}) of the matrices $\alg{L}_n^{-1}\alg{A}_n$.
The individual points in the infinite dimensional spectrum can, however, be approximated
with different speed that is at least linear and uniformly bounded from zero.

Consider an arbitrary point $\mu$ in the spectrum of the operator
$\bc{L}^{-1}\bc{A}$. It should be noted that $\mu$ may not be an eigenvalue,
and that our investigation differs from the eigenvalue problem studied in the numerical PDE literature
which is based on approximations of the eigenvalues within the framework of
infinite dimensional compact solution operators.

%
Using theorem~\ref{th:Laplacian_preconditioning_operator_spectrum}, $\mu$ is the image under
$k(x)$ of some point $y \in \overline{\Omega}$, i.e., $\mu = k(y)$.
\textcolor{black}{We first consider the case $y \in {\Omega}$. The case
$y \in {\partial \Omega}$ will be resolved later by a simple limiting argument.}
Let $\varepsilon > 0$ be an arbitrarily small positive constant, and let
$$ \delta =  \frac{\varepsilon}{2 \sup_{x\in\Omega} \|\nabla k(x)\|},$$
provided that $\sup_{x\in\Omega} \|\nabla k(x)\| > 0$. (The case $\sup_{x\in\Omega} \|\nabla k(x)\| = 0$ is uninteresting because then $\bc{A} = c \bc{L}$ for some constant $c$.)
Consider further a Galerkin discretization such that the support of at least one of
the nodal basis functions\footnote{Supports of all discretization functions are contained in $\overline{\Omega}$.}
that contains the point $y$  is itself contained in the disc with center $y$ and radius $\delta$. Denote this support $\bc{T}_j$ and the associated eigenvalue
of the discretized operator given by theorem~\ref{th:Laplacian_preconditioning_operator_matrix}
as $\lambda_{\pi(j)}$. Using Corollary~3.2 in~\cite{Ger19} (with $y=\hat{x}_j$)
%
\begin{align}
	|\lambda_{\pi(j)} - k(y)|
		&\leq \sup_{ x\in\bc{T}_j} \|k(x) - k(y)\|
			 \nonumber\\
		&\leq \delta \, \|\nabla k(y)\|
			+ \frac{1}{2} \delta^2 \sup_{ x\in\bc{T}_j} \|D^2k( x)\|,
             \label{eq:linear_bound}
\end{align}
where $D^2k(x)$ denotes the second order derivative
of the function $k(x)$\footnote{{See \cite[Section 1.2]{CiaB02} for the definition
{of the second order derivative.}}}. For $\varepsilon$ sufficiently small we thus get, after a simple manipulation,
\begin{equation}\label{eq:h_taylor}
	|\lambda_{\pi(j)} - \mu| \leq \varepsilon.
\end{equation}
\textcolor{black}{If $\mu = k(y)$ and $y \in \partial \Omega$, the same conclusion can be obtained using
the previous derivation and the continuity of $k(x)$ throughout $\overline{\Omega}. $}
Summing up, this proves the following theorem:
			
\begin{theorem}[Approximation of the spectrum by matrix eigenvalues]
\label{th:Spectrum_Approximation_Laplacian_preconditioning}
Let $k(x)$ be twice continuously differentiable and uniformly positive throughout
the closure $\overline{\Omega}$. Let the maximal diameter of the local supports
of the basis functions used in the Galerkin discretization (\ref{eq:support})-(\ref{eq:dis:B})
vanishes as $n \rightarrow \infty$.
Then any point in the spectrum of the operator $\bc{L}^{-1}\bc{A}$ is for $n \rightarrow \infty$
approximated to an arbitrary accuracy by the eigenvalues of the matrices  $\alg{L}_n^{-1}\alg{A}_n$
representing the discretized preconditioned operators.
\end{theorem}

\noindent
The linear part of the upper bound~(\ref{eq:linear_bound}) for the approximation error
is for the individual spectral points $\mu$ proportional to the size of the gradient
$\| \nabla k(y) \|$, where $k(y) = \mu$. Since the size of this gradient is uniformly bounded from above throughout
$\overline{\Omega}$, the speed of convergence towards the individual
spectral points, as the h-refinement proceeds, is uniformly bounded from zero throughout
the whole spectrum of $\bc{L}^{-1}\bc{A}$. It can however differ for different spectral points.

\section{Generalization to \textcolor{black}{(piecewise)} continuous and uniformly positive $g(x)$.}
\label{Preconditioning by a general B}

The purpose of this section is to generalize the results presented above to preconditioners in the form \eqref{eq:B}. We first present the theorems and a corollary, and thereafter their proofs are discussed.

The content of the present section is motivated by the desire to increase our knowledge about second order differential operators and preconditioning issues. In particular, to obtain a better understanding of the benefits of applying piecewise constant preconditioners.

\begin{theorem}[Spectrum of the infinite dimensional preconditioned operator]
\label{th:operator}
\mbox{} \newline Consider an open and bounded Lipschitz domain $\Omega\subset \nmbr{R}^2$.
Assume that the scalar functions  $g(x)$ and $k(x)$ are uniformly positive and continuous
throughout the closure $\overline{\Omega}$.
Then the spectrum of the operator $\bc{B}^{-1}\bc{A}$, defined in \eqref{eq:spectrum}, equals
\begin{equation}\label{eq:result}
 \mathrm{sp}(\mathcal{B}^{-1} \mathcal{A}) = \left[\inf_{x\in\overline{\Omega}}\
 \frac{k(x)}{g(x)},\,\sup_{x\in\overline{\Omega}}\ \frac{k(x)}{g(x)}\right].
\end{equation}
\end{theorem}

\textcolor{black}{The next theorem deals with the localization of the eigenvalues of the preconditioned matrix
arising from the discretization. It does not consider the approximation of the spectrum of the infinite dimensional operator $\bc{B}^{-1} \bc{A}$. Analogously to~\cite{Ger19}, we can therefore relax the assumptions
about the continuity of the coefficient functions $k(x)$ and $g(x)$.}

\begin{theorem}[Eigenvalues of the preconditioned matrix]
\label{th:matrix}
Let $0< \lambda_1\leq\lambda_2\leq\ldots\leq\lambda_n$ be the eigenvalues of
$\alg{B}_n^{-1}\alg{A}_n$, where $\alg{A}_n$ and $\alg{B}_n$ are defined in
\eqref{eq:dis:A} and \eqref{eq:dis:B}, respectively. Let $g(x)$ and $k(x)$ be bounded,
uniformly positive and {\em piecewise continuous} functions. Then there exists a (possibly non-unique)
permutation $\pi$ such that the eigenvalues of the matrix $\alg{B}_n^{-1}\alg{A}_n$ satisfy
\begin{equation}\label{eq:theorem}
\lambda_{\pi(j)} \in \left[{\inf_{ x\in\bc{T}_j}} \frac{k(x)}{g(x)},{\sup_{ x\in\bc{T}_j}}
\frac{k(x)}{g(x)}\right],\quad j = 1,\ldots,n,
\end{equation}
where $\bc{T}_j$ is defined in~(\ref{eq:support}).
\end{theorem}

\begin{corollary}[{Pairing the eigenvalues and} the nodal values]
\label{th:taylor}
Using the notation and the assumptions of Theorem \ref{th:matrix}, consider any point
${\hat{x}_j}\in {\bc{T}_j}$.
Then the associated eigenvalue $\lambda_{\pi(j)}$ of the matrix $\alg{B}^{-1}\alg{A}$ satisfies
\begin{equation}\label{eq:bound:loose}
	|\lambda_{\pi(j)} - r({\hat{x}_j})| \leq{\sup_{ x\in\bc{T}_j}}|r( x) - r( {\hat{x}_j})|,
\quad{ j=1,\ldots, n,}
\end{equation}
where
\[r(x)\equiv\frac{k(x)}{g(x)}.\]
If, in addition,  $k( x)$ and $g(x) \in \bc{C}^2({\bc{T}_j})$, then
\begin{align}\nonumber
	|\lambda_{\pi(j)} - r({\hat{x}_j})|
		&\leq {\sup_{ x\in\bc{T}_j}}|r( x) - r({\hat{x}_j})| \\
			 \label{eq:h_taylor}
		&\leq \hat{h}\|\nabla r( {\hat{x}_j})\|
			+ \tfrac{1}{2}\hat{h}^2{\sup_{ x\in\bc{T}_j}}\|D^2r( x)\|, \quad{ j=1,\ldots, n,}
\end{align}
where $\hat{h} = {\mathrm{diam}(\bc{T}_j)}$ and $D^2r(x)$ denotes the second order derivative
of $r(x)$.
In particular, \eqref{eq:bound:loose} and \eqref{eq:h_taylor} hold for any discretization mesh node
${\hat{x}_j}\in {\bc{T}_j}$.
\end{corollary}
\noindent While the proofs of Theorem \ref{th:operator} and Theorem \ref{th:matrix} are presented below, Corollary \ref{th:taylor} follows immediately by applying
\cite[Corollary 3.2]{Ger19} to the ratio function $r(x)$.

\subsection*{Proof of Theorem \ref{th:operator}.}
Recall the definition \eqref{eq:B} of the operator $\bc{B}$, and let us introduce the following notation for the inner product and norm induced by $\bc{B}$,
\begin{align*}
(u,v)_{\bc{B}} &= \langle \mathcal{B} u, v \rangle, \quad u,  v \in H_0^1(\Omega), \\
\| u \|_{\bc{B}}^2 &= (u,u)_{\bc{B}},  \quad u \in H_0^1(\Omega).
\end{align*}
\begin{enumerate}
\item The proof of the fact that
\[ \mathrm{sp}(\mathcal{B}^{-1} \mathcal{A}) \subset  \left[\inf_{x\in\overline{\Omega}}\
\frac{k(x)}{g(x)},\,\sup_{x\in\overline{\Omega}}\ \frac{k(x)}{g(x)}\right] \]
is analogous to the proof in \cite[Section 3]{Ger20}, employing the inner product induced by $\bc{B}$ instead of that induced by the Laplacian $\bc{L}$.
\textcolor{black}{More precisely, using the self-adjointness of the operator $\bc{B}^{-1}\bc{A}$
 with respect to the inner product $(\cdot,\cdot)_{\bc{B}}$,
%
%
the spectrum
of $\bc{B}^{-1}\bc{A}$ is real and it is contained in the interval}
\begin{align}
\nonumber
 \mathrm{sp}(\mathcal{B}^{-1} \mathcal{A})
&\subset
\left[\inf_{u\in H_0^1(\Omega)} \frac{(\bc{B}^{-1}\bc{A}u,u)_{\bc{B}}}{(u,u)_{\bc{B}}},\
		\sup_{u\in H_0^1(\Omega)} \frac{(\bc{B}^{-1}\bc{A}u,u)_{\bc{B}}}{(u,u)_{\bc{B}}}\right] \\
		\label{NBF0}
		&=
	\left[\inf_{u\in H_0^1(\Omega)} \frac{\langle\bc{A}u,u\rangle}{\langle\bc{B}u,u\rangle},\
		\sup_{u\in H_0^1(\Omega)} \frac{\langle\bc{A}u,u\rangle}{\langle\bc{B}u,u\rangle}\right].
\end{align}
Moreover, the endpoints of this interval are contained in the spectrum.
It remains to bound
\begin{equation}\label{eq:values}
\frac{\langle\bc{A}u,u\rangle}{\langle\bc{B}u,u\rangle}
\end{equation}
in terms of the values of the scalar functions $g(x)$ and $k(x)$.
Let $\| \cdot \|$ denote the standard Euclidean norm. Then,
\begin{align}
\nonumber
\sup_{u\in H_0^1(\Omega)} \frac{\langle\bc{A}u,u\rangle}{\langle\bc{B}u,u\rangle}
	&= \sup_{u\in H_0^1(\Omega)} \frac{\int_{\Omega} k(x) \|\nabla u\|^2}
{\int_{\Omega} g(x)\|\nabla u\|^2}
       = \sup_{u\in H_0^1(\Omega)} \frac{\int_{\Omega} \frac{k(x)}{g(x)} g(x)\|\nabla u\|^2}
{\int_{\Omega} g(x)\|\nabla u\|^2} \\
\label{NBF0.1}
	&\leq \sup_{x\in\Omega}\frac{k(x)}{g(x)},
\end{align}
where we have used the assumption that $k(x)$ and $g(x)$ are uniformly positive and continuous on $\overline{\Omega}$.
Similarly,
\begin{align}
\label{NBF0.2}
\inf_{u\in H_0^1(\Omega)} \frac{\langle\bc{A}u,u\rangle}{\langle\bc{L}u,u\rangle}	
&\geq \inf_{x\in\Omega}\frac{k(x)}{g(x)}.
\end{align}
\item The proof of \textcolor{black}{the converse inclusion}
\[ \left[\inf_{x\in\overline{\Omega}}\ \frac{k(x)}{g(x)},\,\sup_{x\in\overline{\Omega}}\
\frac{k(x)}{g(x)}\right] \subset \mathrm{sp}(\mathcal{B}^{-1} \mathcal{A})\]
is similar to the proof of \cite[Theorem 3.1]{Nie09}.
\begin{itemize}
\item For an arbitrary $x_0 \in \Omega$, consider $\lambda_0 = k(x_0)/g(x_0)$.
\item Let $\{ v_{r} \}_{r \in \mathbb{R}_{+}}$ be a set of functions
satisfying\footnote{Note that no limit of $v_r$, as $r \rightarrow 0$,
  is needed in this proof.  Only the existence of a set of functions
  satisfying \req{NBF1} and \req{NBF2} is required.}
\begin{align}
\label{NBF1}
& {\rm supp}(v_{r}) \subset x_0 + U_{r} , \\
\label{NBF2}
& \| v_{r} \|_{\bc{B}} = 1 ,
\end{align}
%
%
where $U_{r} = \{ z \in \mathbb{R}^2 | \; \| z \| \leq r \}$.
\item  Next, let
\begin{equation}
\label{NBF3}
u_r = (\lambda_0 \bc{I} - \bc{B}^{-1}\bc{A})v_r,
\end{equation}
and observe that
\begin{align*}
\bc{B} u_r = (\lambda_0 \bc{B} - \bc{A})v_r.
\end{align*}
Consequently (see \eqref{eq:B}),
\begin{align*}
\|u_r\|_{\bc{B}}^2 &= \langle (\lambda_0 \bc{B} - \bc{A})v_r, u_r \rangle \\
&= \int_{x_0+U_r} \left( g(x)\lambda_0 - k(x) \right) \nabla v_r \cdot \nabla u_r \\
&\leq \sup_{x \in x_0+U_r} \left| g(x) \frac{k(x_0)}{g(x_0)} - k(x) \right| \, g_{\mathrm{min}}^{-1} \, |(v_r,u_r)_{\bc{B}}| \\
&\leq \sup_{x \in x_0+U_r} \left| g(x) \frac{k(x_0)}{g(x_0)} - k(x) \right| \,
g_{\mathrm{min}}^{-1} \, \| v_r \|_{\bc{B}} \| u_r \|_{\bc{B}},
\end{align*}
\textcolor{black}{where}
$$\textcolor{black}{g_{\mathrm{min}} = \min_{x \in x_0 + U_r} g(x).}$$
\textcolor{black}{Employing \eqref{NBF2},}
\begin{align*}
\|u_r\|_{\bc{B}} \leq \sup_{x \in x_0+U_r} \left| g(x) \frac{k(x_0)}{g(x_0)} - k(x) \right| \, g_{\mathrm{min}}^{-1},
\end{align*}
and from the continuity of $g(x)$ and $k(x)$ we conclude that
\begin{equation}
\label{NBF4}
\lim_{r\to0}\|u_r\|_{\bc{B}}=0.
\end{equation}
\item Assume that $\lambda_0 \notin \mathrm{sp}(\mathcal{B}^{-1} \mathcal{A})$.
Then $\lambda_0 \bc{I} - \bc{B}^{-1}\bc{A}$ has a bounded inverse, and \eqref{NBF3} and \eqref{NBF4}
imply that
\[
\|v_r\|_{\bc{B}} \leq \| (\lambda_0 \bc{I} - \bc{B}^{-1}\bc{A})^{-1} \|_{\bc{B}} \,  \|u_r\|_{\bc{B}}
\longrightarrow 0
\]
as $r \rightarrow 0$, which contradicts \eqref{NBF2}.  We conclude that
\[
\lambda_0 = \frac{k(x_0)}{g(x_0)} \in \mathrm{sp}(\mathcal{B}^{-1} \mathcal{A}).
\]
\item Since $x_0 \in \Omega$ was arbitrary in this argument, and $g$ and $k$ are continuous and uniformly positive on $\overline{\Omega}$, it follows that
\begin{equation}
\label{NBF5}
 \left( \inf_{x\in\overline{\Omega}}\ \frac{k(x)}{g(x)},\,\sup_{x\in\overline{\Omega}}\
\frac{k(x)}{g(x)}\right) \subset \mathrm{sp}(\mathcal{B}^{-1} \mathcal{A}).
\end{equation}
As mentioned above, according to the general \textcolor{black}{result} for self-adjoint operators, the endpoints of the interval in \eqref{NBF0} are contained in the spectrum of $\bc{B}^{-1} \bc{A}$. Combining this with \eqref{NBF0.1} and \eqref{NBF0.2} yield that also the endpoints of the interval \eqref{NBF5} belong to $\mathrm{sp}(\mathcal{B}^{-1} \mathcal{A})$.
\end{itemize}
\end{enumerate}

\subsection*{\textcolor{black}{Proof Theorem \ref{th:matrix}}.}
The proof of Theorem \ref{th:matrix} is analogous to the proof of~Theorem~3.1 in\cite{Ger19}.
As was explained in detail in \cite[Section 3.2]{Ger20_phd}, due to the use of the Hall's theorem for  bipartite graphs (see, e.g., \cite[Theorem 5.2]{Bondy1976}), it is sufficient to prove the statement formulated in the following lemma (cf.~\cite[Lemma 3.3]{Ger19}).

\begin{lemma}\label{th:lemma}
\textcolor{black}{Assume that $k(x)$ and $g(x)$ are uniformly positive, {bounded and piecewise continuous} functions,
and let the matrix $\alg{B}_n^{-1}\alg{A}_n$ be given by \eqref{eq:dis:A} and \eqref{eq:dis:B}.}
Let, moreover, $\bc{J}\subset\{1,\ldots,n\}$ and let $\bc{T}_\bc{J} := \bigcup_{j\in \bc{J}} \bc{T}_j$
be the union of the supports of the basis functions $\phi_j$, $j\in\bc{J}$.
Then there are at least $|\bc{J}|$ eigenvalues of the matrix $\alg{B}_n^{-1}\alg{A}_n$
located in the interval
\begin{equation}
\label{eq:int:ratio}
\left[{\inf_{ x\in\bc{T}_\bc{J}}} \frac{k(x)}{g(x)},{\sup_{ x\in\bc{T}_\bc{J}}}
\frac{k(x)}{g(x)}\right].
\end{equation}
\end{lemma}
\begin{proof}
Following the proof of~Lemma~3.3 in~\cite{Ger19}, consider, for any set \textcolor{black}{of} indices
$\bc{J}\subset \{1,\ldots,n\}$, the (local) perturbation $\tilde{k}_\bc{J}( x) $
of the coefficient function $k(x)$,
\begin{align}
\label{NBF6}
\tilde{k}_\bc{J}( x) &=
	\begin{cases}
		K\cdot g(x) \qquad 	\mbox{for} \quad  x\in\bc{T}_{\bc{J}}, \\
		k( x) \quad \mbox{elsewhere},
	\end{cases}
\end{align}
\textcolor{black}{where $K$ is a positive scalar.}
Analogously to \eqref{eq:dis:A}, the matrix $\tilde{\alg{A}}_{\bc{J},n}$ obtained
by the discretization of the associated perturbed operator $\tilde{\bc{A}}_{\bc{J}}$ is given by
\begin{align*}
[\tilde{\alg{A}}_{\bc{J},n}]_{lj} &=
	\dual{\tilde{\bc{A}}_{\bc{J},h}\phi_j}{\phi_l} =
	\int_{\Omega} \tilde{k}_{\bc{J}}(x) \nabla \phi_j\cdot \nabla\phi_l\,.
\end{align*}
\textcolor{black}{The} simple observation
\[\tilde{\alg{A}}_{\bc{J},n}\alg{e}_j = K\,\alg{B}_n \alg{e}_j,\quad j \in{\bc{J}},\]
shows that $K$ is an eigenvalue of the matrix $\alg{B}_n^{-1}\tilde{\alg{A}}_{\bc{J},n}$
with \textcolor{black}{multiplicity of at least $| \bc{J} |$.}

By similarity transformations, the spectrum of $\alg{B}_n^{-1}\alg{A}_n$ \textcolor{black}{equals}
the spectrum of $\alg{B}_n^{-1/2}\alg{A}_n \alg{B}_n^{-1/2}$, and the spectrum
of $\alg{B}_n^{-1}\tilde{\alg{A}}_{\bc{J},n}$ is equal to the spectrum of
$\alg{B}_n^{-1/2}\tilde{\alg{A}}_{\bc{J},n}\alg{B}_n^{-1/2}$.
Using the standard perturbation result for symmetric matrices
(see, e.g., \cite[Corollary~4.9, p.~203]{1990_SteSu_B}),  there are at least
$|\bc{J}|$ eigenvalues  of $\alg{B}_n^{-1}\alg{A}_n$ in the interval
\begin{equation}
\label{eq:3:bound}
[ K + \theta_{min}, K + \theta_{max} ] \subset  [K - \Theta, K + \Theta ],
\end{equation}
where  $\Theta = \max\{|\theta_{min}|,|\theta_{max}|\}$ and $\theta_{min}$ and
$\theta_{max}$ denote the smallest and largest eigenvalues\textcolor{black}{, respectively,} of the perturbation matrix
$\alg{B}_n^{-1} (\alg{A}_n - \tilde{\alg{A}}_{\bc{J},n})$.

The Rayleigh quotient for {an} eigenvalue-eigenvector pair $(\theta,\alg{q})$ 
of \textcolor{black}{the perturbation} matrix, with the associated eigenfunction
$q = \sum_{j=1}^N\nu_j\phi_j, \, \alg{q}^T = [\nu_1,\ldots,\nu_N]$, satisfies
\begin{align*}
\theta &= \frac{\alg{q}^T(\alg{A}_n - \tilde{\alg{A}}_{\bc{J},n})\alg{q}}{\alg{q}^T\alg{B}_n\alg{q}}
		= \frac{\langle ( \bc{A} - \tilde{\bc{A}}_{\bc{J}})q, q \rangle}{\langle \bc{B}q, q \rangle}
		\\
	   &= \frac{\int_{\Omega}  (k( x) - \tilde{k}_{\bc{J}}( x))\nabla q \cdot\nabla q\, d x }
	   			{\int_{\Omega} g(x)\|\nabla q\|^2\, d x}
		= \frac{\int_{\bc{T}_{\bc{J}}} (k( x) -K g(x))\|\nabla q\|^2\, d x }
				{\int_{\Omega} g(x)\|\nabla q\|^2\, d x} \\
		&= \frac{\int_{\bc{T}_{\bc{J}}} \left(\frac{k( x)}{g(x)} -K\right) g(x)\|\nabla q\|^2\, d x }
				{\int_{\Omega} g(x)\|\nabla q\|^2\, d x}, 		
\end{align*}
where we used the fact that $\tilde{k}_{\bc{J}}( x) = k( x)$ for
$x \in \Omega \setminus \bc{T}_{\bc{J}}$; see \eqref{NBF6}.
Using 
the uniform positivity of $g(x)$,
\begin{equation}
\label{BF1}
|\theta| \leq {\sup_{ x\in\bc{T}_{\bc{J}}}}\left|\frac{k( x)}{g(x)} - K\right|.
\end{equation}
Substituting \eqref{BF1} into \eqref{eq:3:bound} yields the existence of at least
$|\bc{J}|$ eigenvalues of $\alg{B}_{\textcolor{black}{n}}^{-1}\alg{A}_{\textcolor{black}{n}}$ in the interval
\begin{equation}\label{eq:K:bound}
	\left[\, K -   {\sup_{ x\in\bc{T}_{\bc{J}}}}\left|\frac{k( x)}{g(x)} - K\right|,\,
		K +   {\sup_{ x\in\bc{T}_{\bc{J}}}}\left|\frac{k( x)}{g(x)} - K\right|
	\,\right].
\end{equation}
Setting
$
K = \tfrac{1}{2}\left(\inf_{ x\in\bc{T}_{\bc{J}}}\frac{k( x)}{g(x)} +
				 \sup_{ x\in\bc{T}_{\bc{J}}}\frac{k( x)}{g(x)}\right)
$
finishes the proof.
\end{proof}

\section{Abstract setting.}
\label{Abstract Galerkin discretization}
This section investigates numerical approximations of the spectrum of preconditioned linear operators
within an abstract Hilbert space setting; see, e.g.,~\cite{MSB15,HPS20}. Let $V$ be an infinite dimensional
real Hilbert space with the inner product
\begin{equation}\label{eq:iproduct}
( \cdot,\cdot): V\times V \mapsto \nmbr{R}.
\end{equation}
Throughout this text, $V^{\#}$  denotes the dual of $V$ consisting of all linear bounded functionals from $V$ to $\nmbr{R}$, with
the associated duality pairing
$$\langle\cdot,\cdot\rangle: V^\#\times V \mapsto \nmbr{R},$$
and the Riesz map
$$\langle\cdot,\cdot\rangle =: (\tau \cdot, \cdot), \quad \tau:  V^\# \mapsto V. $$

Consider two bounded and coercive linear operators $\bc{A}, \bc{B}: V \mapsto V^{\#}$ that are self-adjoint
with respect to the duality pairing. We will investigate whether all the points in the spectrum
of the preconditioned operator $\bc{B}^{-1}\bc{A}: V\mapsto V$ are approximated to an arbitrary
accuracy by the eigenvalues of the finite dimensional operators in  a sequence $\{\bc{B}_n^{-1}\bc{A}_n\}$
determined via the Galerkin discretization.

\textcolor{black}{For the problems considered in sections \ref{Preconditioning by Laplacian} and \ref{Preconditioning by a general B}, we obtained concrete expressions for the approximations of the spectrum of $\bc{B}^{-1}\bc{A}$ in terms of the coefficient functions $g(x)$ and $k(x)$. Such detailed information is, of course, not obtainable for the present abstract problem setting. We must be content with analyzing whether, or in what sense, the set of eigenvalues of the discretized mapping converges toward the spectrum of the corresponding infinite dimensional operator.}

Since $\bc{B}^{-1}\bc{A}$ is self-adjoint with respect to the inner product
\begin{equation}\label{eq:iproductB}
( \cdot,\cdot)_{\bc{B}}:=  \langle \bc{B} \, \cdot, \cdot \rangle : V\times V \mapsto \nmbr{R},
\end{equation}
it is convenient to use 
this inner product instead of~(\ref{eq:iproduct})
whenever appropriate\footnote{Since $\bc{B}^{-1}\bc{A} = (\tau \bc{B})^{-1} (\tau \bc{A})$, one can
as an alternative investigate approximations of the spectrum of the symmetrized operator
$(\tau\bc{B})^{-1/2} \tau \bc{A} (\tau\bc{B})^{-1/2}$ that is self-adjoint with respect to
the inner product~(\ref{eq:iproduct}).}. The associated norm $\| \cdot\|_{\bc{B}}$  is equivalent
to the norm $\| \cdot \|$ defined by the inner product~(\ref{eq:iproduct}), and the Riesz map $\tau_{\bc{B}}$
representing the operator preconditioning is determined by
\begin{equation}\nonumber
\langle\cdot,\cdot\rangle =: (\tau_{\bc{B}} \, \cdot, \cdot)_{\bc{B}} =
\langle \bc{B}\, \tau_{\bc{B}}\, \cdot, \cdot \rangle, \quad \mbox{i.e.,}
\; \tau_{\bc{B}} = \bc{B}^{-1}.
\end{equation}

The investigated preconditioned operator $\bc{B}^{-1}\bc{A}$ is continuously invertible
on the Hilbert space $V$ of infinite dimension. Therefore its finite dimensional approximations
can not converge to it in norm (uniformly). We will instead use Theorem \ref{th:Kato_(Chatelin)} (below)
that assumes the pointwise (strong) convergence. Its statement reformulates a theorem
presented in~\cite[chapter VIII, \S~1.2, Theorem~1.14, p.~431]{1980Kato_Book}, which is
reproduced also in~\cite[section~5.4, Theorem~5.12, pp.~239-240]{1983Chatelin_Book}.
The second monograph also provides several references to related results of
J.~Descloux and collaborators published earlier; see, in particular,~\cite{DesNasRap78}.
In terms of the spectral representation of self-adjoint operators,
a bit stronger statements were proved in the context of the problem of moments
in~\cite[section~III.3, Theorem IX, p.~61]{VorB65} and, more generally,
in~\cite[chapter VIII, \S~1.2, Theorem~1.15, p.~432]{1980Kato_Book}.
The formulations in~\cite{1980Kato_Book} and~\cite{1983Chatelin_Book} require a careful study
of parts of the books. We will therefore, for the sake of convenience, include
 a proof of the following theorem in the Appendix.

\begin{theorem}[Approximation of the spectrum of self-adjoint operators]
\label{th:Kato_(Chatelin)}
Let $\bc{Z}$ be a linear self-adjoint operator on a Hilbert space $V$
and let $\{\bc{Z}_n\}$ be a sequence of linear self-adjoint operators on $V$
converging to $\bc{Z}$ pointwise (strongly). Then, for any point $\lambda \in \mathrm{sp}(\bc{Z})$
in the spectrum of $\bc{Z}$, and for any of its neighbourhoods, there exists $\bc{Z}_j$
such that the intersection of its spectrum $\mathrm{sp}(\bc{Z}_j)$ with this neighbourhood
is nonempty.
\end{theorem}

Using this theorem and the Hilbert space $V$ equipped with the inner product~(\ref{eq:iproductB}), it remains, within our setting, to prove that the self-adjoint operators $\bc{Z}_n$,
which arise from $\bc{B}_n^{-1}\bc{A}_n$ by extending it to the whole space $V$, converge
pointwise to the original self-adjoint operator
$$\bc{Z} := \bc{B}^{-1}\bc{A}.$$
The discretization will be based on a sequence of subspaces $\{V_n\}, \, V_n \subset V,$
satisfying the  approximation
property~(\ref{eq:approximation_property_of subspaces})\footnote{In the limit equal to zero, i.e.,
it does not matter which of the equivalent norms, $\| \cdot\|_{\bc{B}}$ or $\| \cdot\|$, we use.}
\begin{equation}
\label{NBF7}
\lim_{n \rightarrow \infty} \inf_{v \in V_n} \| w - v \| = 0   \quad \mbox{for all} \;
 w \in V,
\end{equation}
see, e.g., \cite[relation~(8)]{ArnFalWin10}. Note that \eqref{NBF7} typically yields that Galerkin discretizations of
boundary value problems are consistent;
see also~\cite[chapter~9, relation~(9.8)]{MSB15}.

Consider a  basis $\Phi_n = (\phi_1^{(n)}, \dots , \phi_n^{(n)})$
of the $n$-dimensional subspace $V_n \subset V$. Then the Galerkin discretizations
$\bc{A}_n$ and $\bc{B}_n$ of the operators $\bc{A}$ and $\bc{B}$ are determined by (see~\cite[section~4.1]{HPS20} and \cite[chapter~6]{MSB15}),
\begin{equation}\label{eq:discretized_operators}
\langle \mathcal{A}_n w, v \rangle :=  \langle \mathcal{A} w, v \rangle \quad
\mbox{and} \quad
\langle \mathcal{B}_n w, v \rangle :=  \langle \mathcal{B} w, v \rangle,
\quad \mbox{for all} \; w, v \in V_n.
\end{equation}
Their matrix representations are given by
\begin{equation}\label{eq:matrix_A}
\mathbf{A}_n =
\left( \langle \mathcal{A} \phi_j^{(n)} , \phi_i^{(n)} \rangle \right)_{i,j = 1, \dots , n},
\end{equation}
and
\begin{equation}\label{eq:matrix_B}
\mathbf{B}_n =
\left( \langle \mathcal{B} \phi_j^{(n)} , \phi_i^{(n)} \rangle \right)_{i,j = 1, \dots , n}.
\end{equation}

The spectrum of the discretized operator $\bc{B}_n^{-1}\bc{A}_n: V_n \mapsto V_n$
is given by the eigenvalues of its matrix representation $\mathbf{B}_n^{-1}\mathbf{A}_n$.
The operator $\bc{B}_n^{-1}\bc{A}_n$ is self-adjoint with respect to the inner
product~(\ref{eq:iproductB}), and the matrix $\mathbf{B}_n^{-1}\mathbf{A}_n$ is self-adjoint
with respect to the algebraic inner product
$(\mathbf{x}, \mathbf{y})_{\mathbf{B}_n} := \mathbf{y}^* \mathbf{B}_n \mathbf{x}$.

Using the orthogonal projection $$\Pi_{\bc{B}}^n: V \mapsto V_n,$$ where the orthogonality is
determined by the inner product~(\ref{eq:iproductB}), $\bc{B}_n^{-1}\bc{A}_n$ is extended
to the whole space $V$,
\begin{equation}\label{eq:extended_fp_operator}
\bc{Z}_n := \bc{B}_n^{-1}\bc{A}_n \, \Pi_{\bc{B}}^n : V \mapsto V_n \subset V.
\end{equation}

We need to show that $\bc{Z}_n$ is self-adjoint with respect to the inner
product~(\ref{eq:iproductB}). Using, for any $u, w \in V$, the associated orthogonal decompositions
$u = \Pi_{\bc{B}}^n u + u^{\perp}$ and $w = \Pi_{\bc{B}}^n w + w^{\perp}$, we can write
\begin{align*}
\langle \bc{B} \bc{Z}_n u, w \rangle
  &= \langle \bc{B} \bc{B}_n^{-1}\bc{A}_n \, \Pi_{\bc{B}}^n u, w \rangle
   = \langle \bc{B} \bc{B}_n^{-1}\bc{A}_n \, \Pi_{\bc{B}}^n u, \Pi_{\bc{B}}^n w \rangle
       + \langle \bc{B} \bc{B}_n^{-1}\bc{A}_n \, \Pi_{\bc{B}}^n u,  w^{\perp} \rangle \\
  &= \langle \bc{B}_n \bc{B}_n^{-1}\bc{A}_n \, \Pi_{\bc{B}}^n u, \Pi_{\bc{B}}^n w \rangle
   = \langle \bc{A}_n \, \Pi_{\bc{B}}^n u, \Pi_{\bc{B}}^n w \rangle
   = \langle \bc{A}_n \, \Pi_{\bc{B}}^n w, \Pi_{\bc{B}}^n u \rangle \\
  &= \langle \bc{B}_n \bc{B}_n^{-1}\bc{A}_n \, \Pi_{\bc{B}}^n w, \Pi_{\bc{B}}^n u \rangle
   = \langle \bc{B} \bc{B}_n^{-1}\bc{A}_n \, \Pi_{\bc{B}}^n w, \Pi_{\bc{B}}^n u \rangle \\
  &= \langle \bc{B} \bc{B}_n^{-1}\bc{A}_n \, \Pi_{\bc{B}}^n w, \Pi_{\bc{B}}^n u \rangle
       + \langle \bc{B} \bc{B}_n^{-1}\bc{A}_n \, \Pi_{\bc{B}}^n w,  u^{\perp} \rangle  \\
  &= \langle \bc{B} \bc{B}_n^{-1}\bc{A}_n \, \Pi_{\bc{B}}^n w, u \rangle
   = \langle \bc{B} \bc{Z}_n w, u \rangle,
\end{align*}
which gives the self-adjointness.

Summarizing, the sequence of subspaces $\{V_n\}$ determines 
a sequence of self-adjoint operators  $\{\bc{Z}_n\}$ defined on the whole space $V$. The dimension of the ranges of these operators is finite, but increases as $n$ increases. It remains to prove that $\{\bc{Z}_n\}$ converges pointwise to $\bc{Z}$.
\begin{theorem}[Pointwise convergence]
\label{th:pointwise_convergence}
Let $\bc{Z} = \bc{B}^{-1}\bc{A}$ be a linear self-adjoint operator on a Hilbert space $V$,
where $\bc{A}, \bc{B}: V \mapsto V^{\#}$ are bounded and coercive linear operators  that are self-adjoint
with respect to the duality pairing,
and let $\{\bc{Z}_n\}$ be the sequence of linear self-adjoint operators 
defined in~(\ref{eq:extended_fp_operator}). Assume that the sequence of subspaces $\{V_n\}$
satisfy the approximation property~\eqref{NBF7}. 
Then the sequence $\{\bc{Z}_n\}$ converges pointwise (strongly) to $\bc{Z}$, i.e.,
for all $w \in V$
\begin{equation*}
\lim_{n \rightarrow \infty} \|  \bc{Z} w - \bc{Z}_n w \| = 0.
\end{equation*}
\end{theorem}

\begin{proof}
Let $w \in V$ be an arbitrary fixed element and define $$f := \bc{Z} w.$$ Consider a finite
dimensional subspace $V_n \subset V$  and the Galerkin discretization of the equation $\bc{Z} w = f$,
using 
the innerproduct $(\cdot,\cdot)_{\bc{B}}$: Find $w_n \in V_n$ such that
\[
(\bc{Z} w_n, v)_{\bc{B}} = (f, v)_{\bc{B}} \quad \mbox{for all } v \in V_n.
\]

This gives, for all $v \in V_n$,
\begin{align}
\nonumber
0  &= (f - \bc{Z} w_n, v)_{\bc{B}} = (f_n - \bc{Z} w_n, v)_{\bc{B}} + (f - f_n, v)_{\bc{B}} \\
\label{NBF8}
   &= (f_n - \widehat{\bc{Z}_n} w_n, v)_{\bc{B}},
\end{align}
where we used the definition $f_n=\Pi_{\bc{B}}^n f$ for the discretized right hand side, i.e.,
$$(f_n - f, v)_{\bc{B}} = 0 \quad \mbox{for all} \; v \in V_n ,$$
giving
$$\| f - f_n\|_{\bc{B}} = \inf_{g \in V_n}\| f - g\|_{\bc{B}}.$$
The equivalence of the norms induced by the innerproducts~(\ref{eq:iproduct}) and
(\ref{eq:iproductB}) and the approximation property~\eqref{NBF7} 
then assure that
\begin{equation}\label{eq:rhs_convergence}
\lim_{n \rightarrow \infty} \| f - f_n \| = 0 \,.
\end{equation}
The discretized operator $\widehat{\bc{Z}_n}$ is determined by
$$
(\widehat{\bc{Z}_n} u, v)_{\bc{B}} :=(\bc{Z} u, v)_{\bc{B}} \quad
\mbox{for all} \; u, v \in V_n ,
$$
and it is easy to verify that $\widehat{\bc{Z}_n} = \bc{B}_n^{-1}\bc{A}_n$. Indeed,
for all $u, v \in V_n$,
\begin{align*}
(\widehat{\bc{Z}_n} u, v)_{\bc{B}} &=
(\bc{B}^{-1}\bc{A} u, v)_{\bc{B}} \\
  &= \langle \bc{A} u, v \rangle \\
  &= \langle \bc{A}_n u, v \rangle \\
  &= \langle \bc{B}_n \bc{B}_n^{-1} \bc{A}_n  u, v \rangle \\
  &= \langle \bc{B} \bc{B}_n^{-1} \bc{A}_n u, v \rangle \\
  &=  (\bc{B}_n^{-1} \bc{A}_n u, v )_{\bc{B}}.
\end{align*}
The previous considerations remain valid when replacing
$\widehat{\bc{Z}_n}$ by its extension $\bc{Z}_n$ to the whole space $V$ given
in~(\ref{eq:extended_fp_operator}) because
\begin{align*}
(\bc{Z}_n u, v)_{\bc{B}} &= (\bc{B}_n^{-1}\bc{A}_n \, \Pi_{\bc{B}}^n u, v)_{\bc{B}} \\
 &= (\bc{B}_n^{-1}\bc{A}_n u, v)_{\bc{B}} \\
 &= (\widehat{\bc{Z}_n} u, v)_{\bc{B}}
\quad \mbox{for all} \; u, v \in V_n.
\end{align*}
Consequently,
$$
\bc{Z}_n w_n = f_n,
$$
see \eqref{NBF8}.

With the previous considerations we can write
\begin{align}\label{eq:substitution_using_linear_equations}
\bc{Z} w - \bc{Z}_n w
         &= (\bc{Z} w - f) + (f - f_n) + (f_n - \bc{Z}_n w_n) + \bc{Z}_n (w_n - w) \nonumber \\
         &= (f - f_n) +  \bc{Z}_n (w_n - w).
\end{align}
Using~(\ref{eq:rhs_convergence}), the first term vanishes as $n \rightarrow \infty$.
As for the second term, $\bc{Z}_n$ results from the Galerkin discretization of $\bc{Z}$
and therefore its norm is bounded independently of $n$. It remains to examine
the size of the difference $w_n - w$. Following the standard derivation of the C\'{e}a's lemma
(cf., e.g.,~\cite[chapter~9, derivation of the relation~(9.8)]{MSB15}), we have for any
$v \in V_n$
\begin{align*}
\alpha_{\bc{A}} \| w - w_n \|^2
   &\leq \langle \bc{A} (w - w_n),  w - w_n \rangle \\
   &= \langle \bc{B} (\bc{B}^{-1}\bc{A}) (w - w_n),  w - w_n \rangle \\
   &= ((\bc{B}^{-1}\bc{A}) (w - w_n),  w - w_n)_{\bc{B}} \\
   &= ((\bc{B}^{-1}\bc{A}) (w - w_n),  w - v)_{\bc{B}}
       +  ((\bc{B}^{-1}\bc{A}) (w - w_n),  v - w_n)_{\bc{B}} \\
   &= ((\bc{B}^{-1}\bc{A}) (w - w_n),  w - v)_{\bc{B}}
       +  ((f - \bc{Z} w_n),  v - w_n)_{\bc{B}} \\
   &= \langle \bc{A} (w - w_n),  w - v \rangle
    \leq C_{\bc{A}} \| w - w_n \| \| w - v \|.
\end{align*}
Here, $\alpha_{\bc{A}}$ and $C_{\bc{A}}$  are the coercivity and boundedness constants
associated with $\bc{A}$, and in the derivation we used the Galerkin orthogonality \eqref{NBF8}, i.e.,
$((f - \bc{Z} w_n),  v - w_n)_{\bc{B}} = 0$. Consequently, denoting
$\kappa(\bc{A}) := C_{\bc{A}} / \alpha_{\bc{A}}$,
$$ \| w - w_n \| \leq \kappa(\bc{A})\,  \inf_{v \in V_n} \| w - v \| \,.$$
Using again the approximation property~\eqref{NBF7}, 
it follows that the second
term in~(\ref{eq:substitution_using_linear_equations}) also vanishes as $n \rightarrow \infty$,
and the proof is finished.
\end{proof}

\noindent
Theorems~\ref{th:Kato_(Chatelin)} and~\ref{th:pointwise_convergence}
immediately give the final corollary.

\begin{corollary}[Spectral approximation]
\label{cor:spectral_approximation}
Consider an infinite dimensional Hilbert space $V$, its dual $V^{\#}$, and
bounded and coercive linear operators $\bc{A}, \bc{B}: V \mapsto V^{\#}$ that are
self-adjoint with respect to the duality pairing. Consider further a sequence
of subspaces $\{V_n\}$ of $V$ satisfying the approximation
property~\eqref{NBF7}. 

Let the sequences of  matrices $\{\mathbf{A}_n\}$ and  $\{\mathbf{B}_n\}$ be defined
by~(\ref{eq:discretized_operators}) - (\ref{eq:matrix_B}). Then all points in the spectrum
of the preconditioned operator
$$\bc{B}^{-1}\bc{A}: V\mapsto V$$
are approximated  to an arbitrary accuracy by the eigenvalues of the preconditioned matrices
in the sequence $\{\mathbf{B}_n^{-1}\mathbf{A}_n\}.$ \textcolor{black}{That is, for any point $\lambda \in \mathrm{sp}(\bc{\bc{B}^{-1}\bc{A}})$ and any $\epsilon > 0$, there exists $n^*$ such that $\mathbf{B}_{n^*}^{-1}\mathbf{A}_{n^*}$ has an eigenvalue $\lambda^*$ satisfying $|\lambda -\lambda^*| < \epsilon$.}
\end{corollary}

\section{Numerical experiments}
\label{experiments}
We used Matlab's PDE-Toolbox to compute scalars $\lambda$ satisfying
\begin{equation}
\label{C1}
\alg{A}_n \alg{v} = \lambda \alg{B}_n \alg{v},
\end{equation}
where $\alg{A}_n$ and $\alg{B}_n$ denote the stiffness matrices defined in \eqref{eq:dis:A} and \eqref{eq:dis:B}, respectively.
Whereas our theoretical study concerns problems with homogeneous Dirichlet boundary conditions, we employed homogeneous Neumann boundary conditions in the numerical experiments below. This was done for the sake of completeness: One can show, in a straightforward manner, that the results presented in sections \ref{Preconditioning by Laplacian} and \ref{Preconditioning by a general B} also hold in the case of Neumann boundary conditions.

Due to the homogeneous Neumann boundary condition, $\alg{0} = \alg{A}_n \alg{c} = \lambda \alg{B}_n \alg{c} = \alg{0}$ for any constant vector $\alg{c}$ and any scalar $\lambda$. Matlab handles this matter, i.e., \eqref{C1} is solved subject to the constraint that $\alg{v}$ must not belong to the intersection of the nullspaces of $\alg{A}_n$ and $\alg{B}_n$.

The generalized eigenvalues and the nodal values of the function $r(x,y)=k(x,y)/g(x,y)$ are sorted in increasing order in the plots below, $\Omega=(-1,1)^2$, and
\begin{align*}
k(x,y)&=(1+50\exp(-5(x^2+y^2)))(2+sin(x+y)), \\
g(x,y)&=1+50\exp(-5(x^2+y^2)), \\
r(x,y)&=2+\sin(x+y).
\end{align*}

Figure \ref{fig1} shows the generalized eigenvalues \eqref{C1} and the nodal values of $r(x,y)$ computed with the mesh depicted in Figure \ref{fig2}. The results obtained with computations performed on a finer grid is visualized in figures \ref{fig3} and \ref{fig4}. The outcome of these experiments is as one could have anticipated from Theorem \ref{th:operator}, Theorem \ref{th:matrix} and Corollary \ref{th:taylor}.
As expected, we can also observe spreading of the computed eigenvalues over the whole spectral interval.

\begin{figure}[H]
\begin{center}
  \includegraphics[width=10cm]{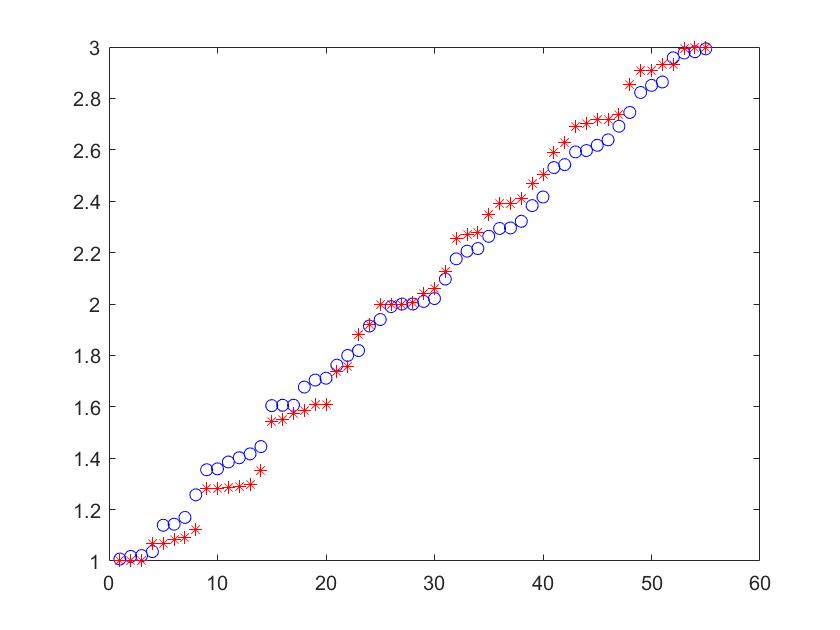}
  \caption{\label{fig1} Generalized eigenvalues (blue circles) and nodal values of $r(x,y)$ (red asterisks) computed with the coarse mesh displayed in Fig. \ref{fig2}.}
\end{center}
\end{figure}

\begin{figure}[H]
\begin{center}
  \includegraphics[width=10cm]{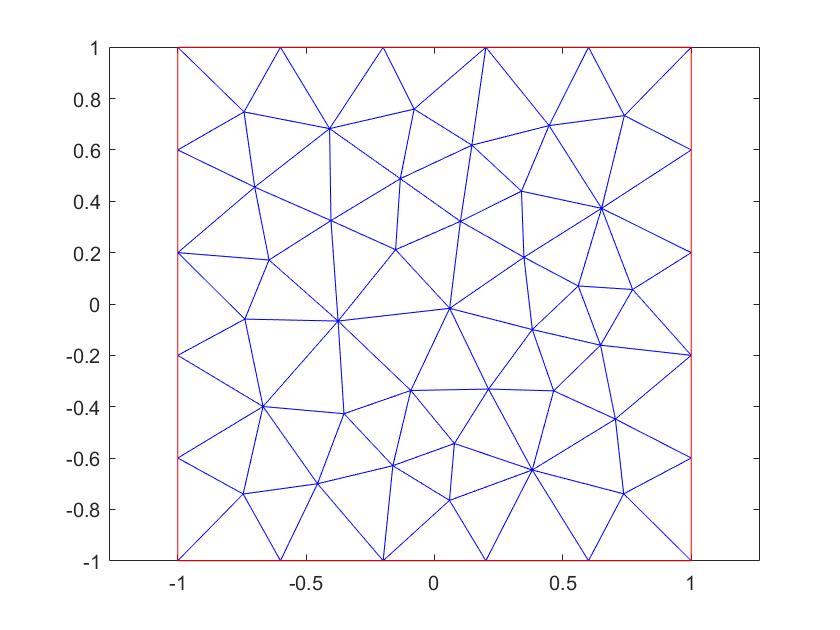}
  \caption{\label{fig2} Matlab's grid, coarse case.}
\end{center}
\end{figure}

\begin{figure}[H]
\begin{center}
  \includegraphics[width=10cm]{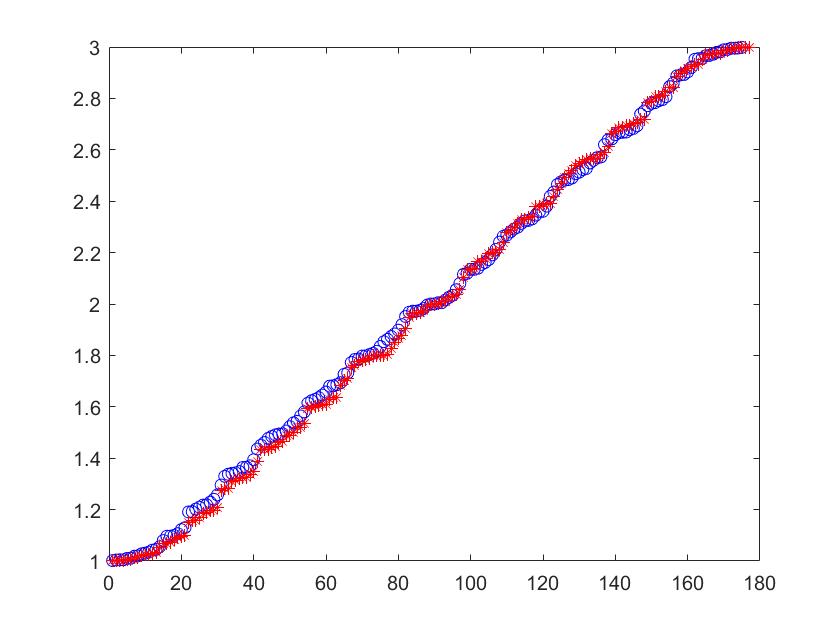}
  \caption{\label{fig3} Generalized eigenvalues (blue circles) and nodal values of $r(x,y)$ (red asterisks) computed with the ``fine'' mesh displayed in Fig. \ref{fig4}.}
\end{center}
\end{figure}

\begin{figure}[H]
\begin{center}
  \includegraphics[width=10cm]{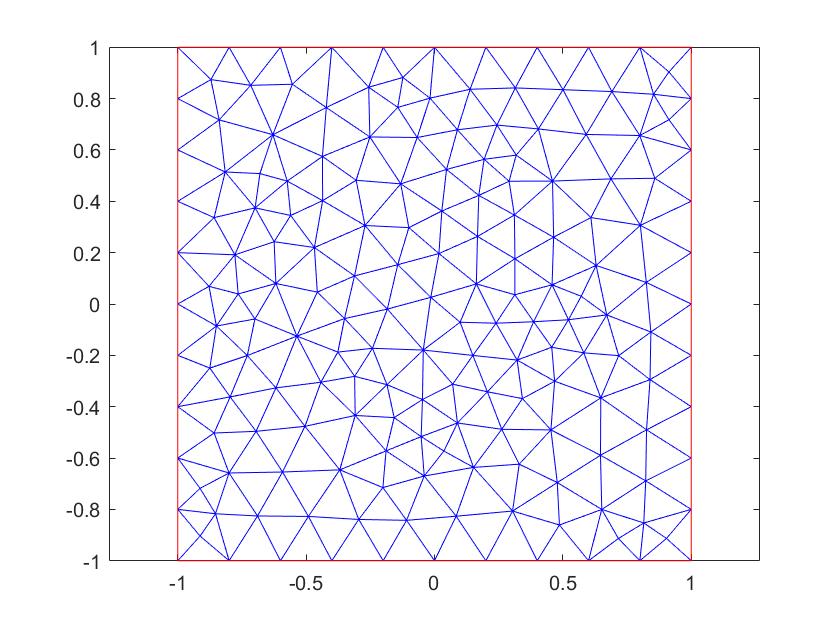}
  \caption{\label{fig4} Matlab's grid, ``fine'' case.}
\end{center}
\end{figure}

\section{Conclusions and further work.}
\label{Conclusions}
We have not only extended our earlier results \cite{Nie09,Ger19,Ger20}, addressing Laplacian preconditioning,  to preconditioners defined in terms of more general elliptic differential operators, but we have also proved that the {\em entire} spectrum of any operator in the form $\bc{B}^{-1} \bc{A}$ can be approximated with arbitrary accuracy by the eigenvalues of the associated  discretized mappings. Here, $\bc{A, B}: V\mapsto V^{\#}$ are linear, bounded, coercive and self-adoint operators defined on an infinite dimensional Hilbert space $V$, and $V^{\#}$ denotes the dual space. Our analysis  differs significantly from the classical investigations of the point spectrum of second order differential operators, which is typically done within the framework of compact (solution) operators.

In our opinion, these results yield a new perspective on the continuous spectrum of preconditioned elliptic differential operators, and there are several unanswered questions: For example, are the results presented in section \ref{Preconditioning by a general B} also valid if the coefficient functions $k(x)$ and $g(x)$ are replaced by symmetric and uniformly positive definite conductivity tensors $K(x)$ and $G(x)$, respectively? Also, and perhaps even more interesting, do Theorem \ref{th:pointwise_convergence} and Corollary \ref{cor:spectral_approximation} also hold for more general continuously invertible operators $\bc{A}: V\mapsto V^{\#}$, e.g., for saddle point problems?


\section*{Acknowledgments.}
The authors thank David Krej\v{c}i\v{r}\'{\i}k, Josef M\'{a}lek and Ivana Pultarov\'{a}
for stimulating discussions during the work on this paper.

\appendix
\section{Approximations of the spectrum of self-adjoint operators.}
\label{Appendix}

Using~\cite[Chapter~3]{1983Chatelin_Book} and~\cite[Chapter~8]{1980Kato_Book}, we first recall 
several results concerning the convergence of linear self-adjoint operators defined on infinite dimensional Hilbert spaces.
By the Hellinger-Toeplitz theorem (see~\cite[Theorem~5.7.2, p.~260]{2013Ciarlet_Book}), any linear self-adjoint
operator $\bc{Z}: V \mapsto V$ on a Hilbert space $V$ is closed and, according to the Banach closed-graph theorem, bounded (continuous).

Consider a bounded linear operator $\bc{G}: V \mapsto V$ (not necessarily self-adjoint, therefore we
for the moment change the notation) and a sequence of its bounded linear approximations
$\{\bc{G}_n \}, \bc{G}_n: V \mapsto V$, that can converge to $\bc{G}$ in different ways:
\begin{itemize}

\item
pointwise (strongly), i.e., \, $\bc{G}_n \stackrel{p}{\rightarrow} \bc{G}$
$$ \mbox{iff for all} \; x \in V, \quad \lim_{n \rightarrow \infty} \| \bc{G} x -  \bc{G}_n x\| = 0 \,;$$

\item
uniformly (in norm), iff \; $\lim_{n \rightarrow \infty} \| \bc{G}  -  \bc{G}_n \| = 0 \,;$

\item
stably, i.e. \, $\bc{G}_n \stackrel{s}{\rightarrow} \bc{G}$  \, iff

  \begin{itemize}
    \item $\bc{G}_n \stackrel{p}{\rightarrow} \bc{G}$, and
    \item the inverse operators \, $\{ \bc{G}_n^{-1} \}$ \, are uniformly bounded, i.e., for some $C > 0$,  $\|\bc{G}_n^{-1}\| \leq C$ for all $n$.
  \end{itemize}

\end{itemize}

\noindent
Clearly, uniform convergence implies pointwise convergence, but the converse implication
does not hold. Since the class of compact operators is closed with respect to uniform convergence,
the uniform convergence concept can not be used to investigate the convergence of compact to non-compact operators, such as
to bounded continuously invertible operators defined on infinite dimensional Hilbert spaces.

The spectral theory for bounded linear operators is based on the concept of the operator resolvent
$$ \bc{R}(\mu) := (\mu \mathcal{I} - \mathcal{G})^{-1}$$ and on the resolvent set
\begin{equation}\label{eq:resolvent_set}
\rho(\bc{G})
:= \left\{ \mu \in \mathbb{C}; \,  \mu \mathcal{I} - \mathcal{G} \;
\mbox{has a bounded inverse} \right\}.
\end{equation}
It is interesting to notice that, for any $\mu \in \rho(\bc{G})$, a sequence of shifted operators
$\{\mu \mathcal{I} - \mathcal{G}_n\}$ converge to $\mu \mathcal{I} - \mathcal{G}$ stably if, and only if,
$\{\mathcal{G}_n\}$  and the resolvents
$\{\bc{R}_n(\mu)\}, \bc{R}_n(\mu):= (\mu \mathcal{I} - \mathcal{G}_n)^{-1}$,
converge to $\mathcal{G}$ and $\bc{R}(\mu)$ pointwise, respectively,
i.e.,\footnote{See~\cite[Lemma~3.16]{1983Chatelin_Book}.}
\begin{equation}\label{stable_pointwise_convergence}
\mu \mathcal{I} - \mathcal{G}_n \stackrel{s}{\rightarrow} \mu \mathcal{I} - \mathcal{G}
\quad \mbox{iff} \quad
\mathcal{G}_n \stackrel{p}{\rightarrow} \mathcal{G} \; \mbox{and} \;
\bc{R}_n(\mu) \stackrel{p}{\rightarrow} \bc{R}(\mu).
\end{equation}
Indeed, using the resolvent identity
$$
\bc{R}_n(\mu) - \bc{R}(\mu) = \bc{R}_n(\mu) \, (\bc{G}_n - \bc{G}) \, \bc{R}(\mu),
$$
the right implication follows immediately from the definition of stable convergence. Conversely, from
the pointwise  convergence of $\bc{R}_n(\mu)$  and the uniform boundedness principle (Banach–Steinhaus theorem) we conclude that $\{\bc{R}_n(\mu)\}=\{ (\mu \mathcal{I} - \mathcal{G}_n)^{-1}  \}$ is uniformly bounded and the result follows.

We will now present the proof of Theorem~\ref{th:Kato_(Chatelin)};
cf. also~\cite[chapter~VIII, \S~1.2, Theorem~1.14, p.~431]{1980Kato_Book} and
\cite[chapter~5, section~4, Theorem~5.12, p.~239-240]{1983Chatelin_Book}.

\begin{theorem}[Approximation of the spectrum of self-adjoint operators]
\label{}
Let $\bc{Z}$ be a linear self-adjoint operator on a Hilbert space $V$
and let $\{\bc{Z}_n\}$ be a sequence of linear self-adjoint operators on $V$
converging to $\bc{Z}$ pointwise (strongly). Then, for any point $\lambda \in \mathrm{sp} (\bc{Z})$
in the spectrum of $\bc{Z}$, and for any of its neighbourhoods, there exists $\bc{Z}_j$
such that the intersection of its spectrum $\mathrm{sp} (\bc{Z}_j)$ with this neighbourhood
is nonempty.
\end{theorem}

\begin{proof}
Consider any point $\lambda \in \mathrm{sp}(\bc{Z}) \subset \nmbr{R}$ in the spectrum of $\bc{Z}$.
Then, for any $\varepsilon > 0$, the point $\mu :=  \lambda + \iota \varepsilon$, where $\iota$
is the complex unit, belongs to the resolvent set $\rho(\bc{Z})$ (because $\bc{Z}$ is self-adjoint). For self-adjoint operators, the norm
of the resolvent at any point in the resolvent set is equal to the inverse of the distance of the given
point to the spectrum (see, e.g.,~\cite[Proposition~2.32]{1983Chatelin_Book}). Therefore, for all $n$,
\begin{align}
\label{NBF9}
\| \bc{R}(\mu) \| &:=
\|(\mu \mathcal{I} - \mathcal{Z})^{-1}\| = \frac{1}{\mathrm{dist}(\mu, \mathrm{sp}(\bc{Z}))}
= \frac{1}{\varepsilon}\,,\\
\label{NBF10}
\| \bc{R}_n(\mu) \| &:=
\|(\mu \mathcal{I} - \mathcal{Z}_n)^{-1}\| = \frac{1}{\mathrm{dist}(\mu, \mathrm{sp}(\bc{Z}_n))}
\leq \frac{1}{\varepsilon}\,.
\end{align}
The inequality in \eqref{NBF10} follows from the assumption that $\{\bc{Z}_n\}$ is a sequence of self-adjoint operators, i.e., $\mathrm{sp}(\bc{Z}_n) \subset \nmbr{R}$.
Note that inequality \eqref{NBF10} provides the uniform boundedness of $\{ \bc{R}_n(\mu) \}$, which, together with the pointwise convergence of $\{\bc{Z}_n\}$, yields that
$$
\mu \bc{I} - \bc{Z}_n \stackrel{s}{\rightarrow} \mu \bc{I} - \bc{Z}.
$$

Using~(\ref{stable_pointwise_convergence}), we thus have the pointwise convergence
of $\{\bc{R}_n(\mu) \}$, i.e., for any $x \in V, ||x\| = 1,$
$$
 \| \bc{R}(\mu) x \|  = \lim_{n \rightarrow \infty} \| \bc{R}_n(\mu) x \|.
$$
Consider a fixed $x \in V, ||x\| = 1,$ such that
$$
\frac{1}{2 \varepsilon} \leq \| \bc{R}(\mu) x \|  \leq \frac{1}{\varepsilon}.
$$
(The existence of such a $x \in V$ follows from \eqref{NBF9}.)
Then, from the pointwise convergence, there must exist $n$ such that
$$
  \| \bc{R}_n(\mu)\|  \geq \| \bc{R}_n(\mu) x \|  \geq \frac{1}{3 \varepsilon}.
$$
Recall that $\| \bc{R}_n(\mu) \|=\mathrm{[dist}(\mu, \mathrm{sp}(\bc{Z}_n))]^{-1}$.
Therefore there exists a point $\lambda_n \in \mathrm{sp}(\bc{Z}_n)$
such that
$$
|\lambda_n - \mu| \leq 3 \varepsilon \quad
\mbox{and, consequently,} \quad
|\lambda_n - \lambda| \leq 4 \varepsilon,
$$
which finishes the proof.
\end{proof}

\bibliographystyle{plain}
\bibliography{references}

\begin{thebibliography}{10}

\bibitem{ArnFalWin10}
D.~N. Arnold, R.~S. Falk, and R.~Winther.
\newblock Finite element exterior calculus: from {H}odge theory to numerical
  stability.
\newblock {\em Bull. Amer. Math. Soc. (N.S.)}, 47:281--354, 2010.

\bibitem{Bondy1976}
J.~A. Bondy and U.~S.~R. Murty.
\newblock {\em Graph theory with applications}.
\newblock American Elsevier Publishing Co., Inc., New York, 1976.

\bibitem{1983Chatelin_Book}
Françoise Chatelin.
\newblock {\em Spectral approximation of linear operators}.
\newblock Academic Press, New York, 1983.

\bibitem{CiaB02}
Philippe~G. Ciarlet.
\newblock {\em The finite element method for elliptic problems}, volume~40 of
  {\em Classics in Applied Mathematics}.
\newblock Society for Industrial and Applied Mathematics (SIAM), Philadelphia,
  PA, 2002.
\newblock Reprint of the 1978 original [North-Holland, Amsterdam; MR0520174 (58
  \#25001)].

\bibitem{2013Ciarlet_Book}
Philippe~G. Ciarlet.
\newblock {\em Linear and nonlinear functional analysis with applications}.
\newblock Society for Industrial and Applied Mathematics, Philadelphia, PA,
  2013.

\bibitem{DesNasRap78}
J.~Descloux, N.~Nassif, and J.~Rappaz.
\newblock On spectral approximation. part~1. the problem of convergence.
\newblock {\em RAIRO - Analyse numérique}, 12:97--112, 1978.

\bibitem{Ger19}
T.~Gergelits, K.~A. Mardal, B.~F. Nielsen, and Z.~Strako{\v{s}}.
\newblock Laplacian preconditioning of elliptic {PDE}s: Localization of the
  eigenvalues of the discretized operator.
\newblock {\em SIAM Journal on Numerical Analysis}, 57(3):1369--1394, 2019.

\bibitem{Ger20}
T.~Gergelits, B.~F. Nielsen, and Z.~Strako{\v{s}}.
\newblock Generalized spectrum of second order differential operators.
\newblock {\em SIAM Journal on Numerical Analysis}, 58(4):2193--2211, 2020.

\bibitem{Ger20_phd}
Tom{\'a}{\v s} Gergelits.
\newblock {\em Krylov Subspace Methods: Analysis and Applications}.
\newblock PhD thesis, Charles University, 2020.

\bibitem{HPS20}
I.~Pultarov\'{a} J.~Hrn\v{c}\'{\i}\v{r} and Z.~Strako{\v{s}}.
\newblock Decomposition of subspaces preconditioning: abstract framework.
\newblock {\em Numerical Algorithms}, 83:57--98, 2020.

\bibitem{1980Kato_Book}
Tosio Kato.
\newblock {\em Perturbation Theory for Linear Operators}.
\newblock Springer-Verlag, Berlin Heidelberg, 1980.

\bibitem{LPZ20}
I.~Pultarov\'{a} M.~Ladeck\'{y} and J.~Zeman.
\newblock Guaranteed two-sided bounds on all eigenvalues of preconditioned
  diffusion and elasticity problems solved by the finite element method.
\newblock {\em Applications of Mathematics}, 2020.

\bibitem{MSB15}
Josef M{\'{a}}lek and Zden{\v{e}}k Strako{\v{s}}.
\newblock {\em Preconditioning and the conjugate gradient method in the context
  of solving pdes}.
\newblock SIAM Spotlight Series, December 2014.

\bibitem{Nie09}
B.~F. Nielsen, A.~Tveito, and W.~Hackbusch.
\newblock Preconditioning by inverting the {L}aplacian; an analysis of the
  eigenvalues.
\newblock {\em IMA Journal of Numerical Analysis}, 29(1):24--42, 2009.

\bibitem{1990_SteSu_B}
G.~W. Stewart and Ji~Guang Sun.
\newblock {\em Matrix perturbation theory}.
\newblock Computer Science and Scientific Computing. Academic Press, Inc.,
  Boston, MA, 1990.

\bibitem{VorB65}
Yu.~V. Vorobyev.
\newblock {\em Methods of Moments in Applied Mathematics}.
\newblock Translated from the Russian by Bernard Seckler. Gordon and Breach
  Science Publishers, New York, 1965.

\end{thebibliography}

\end{document}